\documentclass[onefignum,onetabnum]{siamonline181217}

\usepackage{lipsum}
\usepackage{amsfonts}
\usepackage{graphicx}
\usepackage{epstopdf}
\usepackage{mathtools}
\usepackage{amssymb}  %
\usepackage{amsmath}
\usepackage{hyperref}
\usepackage{wrapfig}
\usepackage{amsopn}
\usepackage{mathrsfs}
\usepackage{calligra}
\usepackage{subcaption}
\usepackage{pgfplots}
\pgfplotsset{compat=1.11}

\usepackage{pgfplots}
\pgfplotsset{compat=1.6}

\usepackage{tikz}
\usetikzlibrary{chains,decorations.pathmorphing,positioning,fit}
\usetikzlibrary{decorations.shapes,calc,backgrounds}
\usetikzlibrary{decorations.text,matrix}
\usetikzlibrary{arrows,shapes.geometric,shapes.symbols,scopes}
\usetikzlibrary{patterns}
\usetikzlibrary{shadows}
\usetikzlibrary{chains,decorations.pathmorphing,positioning,fit}
\usetikzlibrary{decorations.shapes,calc,backgrounds}
\usetikzlibrary{decorations.text,matrix}
\usetikzlibrary{arrows,shapes.geometric,shapes.symbols,scopes}
\usetikzlibrary{shapes.misc}

\ifpdf
  \DeclareGraphicsExtensions{.eps,.pdf,.png,.jpg}
\else
  \DeclareGraphicsExtensions{.eps}
\fi

\usepackage{enumitem}
\setlist[enumerate]{leftmargin=.5in}
\setlist[itemize]{leftmargin=.5in}

\newsiamremark{example}{Example}
\newsiamremark{remark}{Remark}
\newsiamremark{hypothesis}{Hypothesis}
\crefname{hypothesis}{Hypothesis}{Hypotheses}
\newsiamthm{claim}{Claim}
\newsiamremark{conjecture}{Conjecture}

\headers{Convex Quaternary Quartics Are Sum of Squares}{B. El Khadir}

\title{On Sum of Squares Representation of Convex Forms \\and Generalized Cauchy-Schwarz Inequalities}

\author{Bachir El Khadir\thanks{The author is with the department of Operations Research and Financial Engineering at Princeton University (\email{bkhadir@princeton.edu}, \url{http://bachirelkhadir.com}).
    This work was partially supported by the DARPA Young Faculty Award, the Princeton SEAS Innovation Award, the NSF CAREER Award, and the MURI Award of the AFOSR.
  }}

\newcommand{\ix}{\mathbf{x}}
\newcommand{\iy}{\mathbf{y}}
\newcommand{\iz}{\mathbf{z}}

\newcommand{\iu}{\mathbf{u}}

\newcommand{\iv}{\mathbf{v}}

\newcommand{\ie}[1]{\mathbf{e_{#1}}}

\newcommand{\RE}[1]{\operatorname{Re}{(#1)}}
\newcommand{\IM}[1]{\operatorname{Im}{(#1)}}
\newcommand{\FORMS}[2]{H_{#1,#2}}
\newcommand{\PSD}[2]{P_{#1,#2}}
\newcommand{\SOS}[2]{\Sigma_{#1,#2}}
\newcommand{\CONVEX}[2]{C_{#1,#2}}

\def\R{\mathbb R}
\def\C{\mathbb C}
\def\N{\mathbb N}
\def\iim{i}

\DeclareMathOperator{\cone}{cone}

\ifpdf
\hypersetup{
  pdftitle={Convex Quaternary Quartics Are Sum of Squares},
  pdfauthor={Bachir El Khadir}
}
\fi

\begin{document}

\maketitle

\begin{abstract}
A convex form of degree larger than one is always nonnegative since it vanishes together with its gradient at the origin. In 2007, Parrilo asked if convex forms are always sums of squares. A few years later, Blekherman answered the question in the negative by showing through volume arguments that for high enough number of variables, there must be convex forms of degree as low as $4$ that are not sums of squares. Remarkably, no examples are known to date. In this paper, we show that all convex forms in $4$ variables and of degree $4$ are sums of squares. We also show that if a conjecture of Blekherman related to the so-called Cayley-Bacharach relations is true, then the same statement holds for convex forms in $3$ variables and of degree $6$. These are the two minimal cases where one would have any hope of seeing convex forms that are not sums of squares (due to known obstructions). A main ingredient of the proof is the derivation of certain ``generalized Cauchy-Schwarz inequalities'' which could be of independent interest.
\end{abstract}



\section{Introduction and Main Result}

The set $\FORMS{n}{k}$ of homogeneous real polynomials (forms) in $n$ variables and of degree $k$ is a central subject of study in algebraic geometry. When the degree $k \eqqcolon 2d$ is even, three convex cones inside $\FORMS{n}{k}$ have received considerable interest. %
The cone of {\it nonnegative} forms
\[\PSD{n}{2d} \coloneqq \{p \in \FORMS{n}{2d} \; | \; p(\ix) \ge 0 \; \forall \ix \in \R^n\},\]
the cone of {\it sum of squares (sos)} forms
\[\SOS{n}{2d} \coloneqq \{p \in \FORMS{n}{2d} \; | \; p = \sum_i q_i^2 \text{ for some forms $q_i \in \FORMS{n}{d}$}\},\]
and the cone of {\it convex} forms
\[\CONVEX{n}{2d} \coloneqq \{p \in \FORMS{n}{2d} \; | \; \nabla^2p(\ix) \succeq 0 \; \forall \ix \in \R^n  \},\]
where $\nabla^2 p(\ix)$ stands for the Hessian of the form $p$ at $\ix$, and the symbol $\succeq$ stands for the partial ordering generated by the cone of positive semidefinite matrices.


The systematic study of the interplay between the cones $\PSD{n}{2d}$ and $\SOS{n}{2d}$ was undertaken by Hilbert at the end of the nineteenth century, when he showed that these two cones are different unless $n \le 2$, $2d \le 2$ or $(n,2d) = (3,4)$ \cite{hilbert_1888}. Even though Hilbert's work provided a strategy for constructing nonnegative forms that are not sos for the smallest number of variables and degrees possible (i.e., forms in $\PSD{3}{6}\setminus \SOS{3}{6}$ and $\PSD{4}{4}\setminus \SOS{4}{4}$), it took almost eighty years for the first explicit examples of such forms to be found by Motzkin and Robinson~\cite{arith_geom_ineq_motzkin_1967, nonnegative_not_sos_robinson_1969, reznick2000some}. See \cite{reznick_hilberts_2007} for a more thorough discussion of the history of this problem.

The relationship between $\CONVEX{n}{2d}$ and $\SOS{n}{2d}$ is much more complicated and it was an open problem for some time whether $\CONVEX{n}{2d} \subseteq \SOS{n}{2d}$ for all $n$ and $d$. (The reverse inclusion is of course false; e.g., $x^2y^2 \in \SOS{2}{4} \setminus \CONVEX{2}{4}$.) Note however, that we trivially have $\CONVEX{n}{2d} \subseteq \PSD{n}{2d}$ since a global minimum of a convex form is always at the origin where the form and its gradient vanish.

Studying the gap between the cone of sum of squares and the cone of convex polynomials is particularly relevant from an optimization point of view, where it is more naturally formulated in the {\it non-homogeneous} setting\footnote{While the properties of being nonnegative and sum of squares are preserved under the homogenization operation $p(\ix) \rightarrow y^{\operatorname{deg}(p)} p(\frac{\ix}{y})$, the property of convexity is not in general. For instance, the polynomial $x^2-1$ is convex, but its homogenization $x^2 - y^2$ is not.}.  Consider the problem of finding the minimum value $p^*$ that a convex polynomial $p$ takes on $\R^n$. Observe that
\[p^* = \max_{\gamma \in \R} \gamma \text{ s.t. } p - \gamma \text{ is nonnegative}.\]
The well-known machinery of ``sum of squares relaxation'' \cite{parrilo_minimizing_convex_2003, parrilo_sos_relaxation_2003}, for which efficient algorithms based on semidefinite programming exist, can be readily applied here to obtain a lower bound $p^{\text{sos}}$ on $p^*$:
\[p^{\text{sos}} \coloneqq \max_{\gamma \in \R} \gamma \text{ s.t. } p - \gamma \text{ is sos}.\]
Note that for any scalar $\gamma$, the polynomial $p - \gamma$ is convex. If we knew a priori that $p-\gamma$ is sos whenever it is nonnegative, then this relaxation becomes exact; i.e., $p^* = p^{\text{sos}}$.


Blekherman has recently shown that for any fixed degree $2d \ge 4$, as the number of variables $n$ goes to infinity, one encounters considerably more convex forms than sos forms \cite{blekherman_convex_not_sos_2009}. Remarkably however, there is not a single known example of a convex form that is not sos. Due to Hilbert's characterization of the cases of equality between the cone of nonnegative forms and the cone of sos forms, the smallest cases where one could have hope of finding such an example correspond to quaternary quartics ($(n,2d) = (4,4)$) and ternary sextics ($(n,2d) = (3, 6)$).
The goal of this paper
prove that no such example exists among quaternary quartics.

\begin{theorem}\label{thm:convex_fourvar_quads_are_sos}
  Every convex quarternary quartic is sos, i.e., $\CONVEX{4}{4} \subseteq \SOS{4}{4}$.
\end{theorem}
Furthermore, we show that if a conjecture of Blekherman related to the so-called Cayley-Bacharach relations is true, no convex form which is not sos can exist among ternary sextics either, i.e., $\CONVEX{3}{6} \subseteq \SOS{3}{6}$.  

A possible plan of attack to show that a convex form is sos is to show that it is {\it sos-convex}. This concept, introduced by Helton and Nie \cite{helton_sosconvexity_2010}, is an algebraic sufficient condition for convexity which also imples the property of being sos. 
 This plan would not be successful for our purposes however, since there exist explicit examples of convex forms that are not sos-convex for both cases $(n,2d) = (4,4)$ and $(n,2d) = (3,6)$ \cite{ahmadi2013complete}. In fact, the problem of characterizing for which degrees $2d$ and number of variables $n$ sos-convexity is also a necessary condition for convexity, has been completely solved in \cite{ahmadi2013complete}. The authors prove that this is the case if and only if $n \le 2$, $2d \le 2$ or $(n, 2d) = (3, 4)$, i.e., the same cases for which $\PSD{n}{2d} = \SOS{n}{2d}$ as characterized by Hilbert, albeit for different reasons. 



 Our proof strategy relies instead on an equivalence due to Blekherman \cite{blekherman_nonnegative_2012} between the membership $p \in \SOS44$ for a nonnegative  form $p$, and the following bounds on point evaluations of the form $p$:
 \begin{equation}
 \sqrt{p(\iu_1)} \le \sum_{i=2}^8  \sqrt{p(\iu_i)} \text{ and } \sqrt 2 \sqrt{|p(\iz)| + \RE{p(\iz)}} \le \sum_{i=3}^8 \sqrt{p(\iv_i)},\label{eq:bounds_cayley_bacharach} 
\end{equation}
 where the real vectors $\iv_i$ and $\iu_i$ are the complex vector $\iz$ come from intersections of quadratic forms (see \cref{thm:caracterization-sos} for a more precise statement). This equivalence is explained in \cref{sec:sos_vs_psd}.
 We show in \cref{sec:proof_main_theorem} that any quaternary quartic form $p$ that satisfies the two inequalities
\begin{equation}\label{eq:cs_quartics}
    Q_p(\ix, \iy) \le \sqrt{p(\ix) p(\iy)} \quad \forall \; \ix, \iy \in \R^4
    \quad \text{ and }\quad
    |p(\iz)| \le Q_p(\iz, \bar \iz) \quad    \forall \; \iz \in \C^4,
  \end{equation}
  where $Q_p(\ix, \iy) \coloneqq \frac1{12} \iy^T\nabla^2p(\ix)\iy$, also satisfies these bounds. 
  These inequalities can be thought of as a generalization of the Cauchy-Schwarz inequality, valid for any $n \times n$ positive semidefinite matrix $Q$:
  \[\ix^TQ\iy \le \sqrt{\ix^TQ\ix \cdot \iy^TQ\iy} \quad \forall \; \ix, \iy \in \R^n.\]
  We show that convex quaternary quartic forms satisfy the inequalities in \cref{eq:cs_quartics}, and are therefore sos. 
  In fact, in \cref{sec:generalized_cauchy_schwarz}, we present generalizations of the Cauchy-Schwarz inequality that apply to convex forms of any degree and any number of variables. We believe that these inequalities could be of independent interest. In \cref{sec:generalization}, we discuss a possible extension of our proof technique to the case of ternary sextics.

\section{Background and Notation}

%

We denote the set of positive natural numbers, real numbers, and
complex numbers by $\mathbb N$, $\R$, and $\C$ respectively. We denote
by $(\ie1, \ldots, \ie n)$ the canonical basis of $\R^n$. We denote by
$i$ the imaginary number $\sqrt{-1}$, and by $\bar z$, $|z|$,
$\RE{z}$, and $\IM{z}$ the complex conjugate, the modulus, the real part
and the imaginary part of a complex number $z$ respectively.

\subsection{Notation for differential operators}
We denote by $\partial_{\iu}$ the partial differentiation operator in
the direction of $\iu \in \mathbb \C^n$, i.e., $\partial_\iu p(\ix)$
is the limit of the ratio $(p(\ix + t\iu) - p(\ix)) / t$ as
$t \rightarrow 0$ for all $n\text{-variate}$ polynomial functions $p$
and all vectors $\ix \in \C^n$.  The gradient operator
$(\partial_{\ie1}, \dots, \partial_{\ie n})^T$ is denoted by $\nabla$,
the Hessian operator $\nabla\nabla^T$ is denoted by $\nabla^2$, and the
Laplacian operator $\partial_{\ie 1}^2 + \dots + \partial_{\ie n}^2$
is denoted by $\Delta$. For a form $p \in \FORMS{n}{2d}$ and vectors
$\ix_1, \dots, \ix_n \in \C^n$, we denote by
$p(\partial_{\ix_1}, \dots, \partial_{\ix_n})$ the differential
operator obtained by replacing the indeterminate $x_k$ with
$\partial_{\ix_k}$ for $k=1,\dots,n$ in the expression
$p(x_1, \dots, x_n)$. We note that taking $k$ partial derivatives of a
$k\text{-degree}$ form results in a constant function. As a
consequence, we consider the quantity
$p(\partial_{\ix_1}, \dots, \partial_{\ix_n})q$ to be a scalar for all
forms $p$ and $q$ in $\FORMS{n}{k}$.

\subsection{Euler's identity}
Euler's identity (see e.g., \cite{opt_hom_lasserre_2002}) links the value that a form $p \in \FORMS{n}{k}$ takes to its gradient as follows :
\[k \; p(\ix) =  \ix^T \nabla p(\ix)  \quad \forall \ix \in \R^n.\]
By applying this identity to the entries of the gradient $\nabla p$, one obtains the following relationship between a form and its Hessian:
\[k(k-1) \;  p(\ix) = \ix^T \nabla^2 p(\ix) \ix \quad \forall \ix \in \R^n.\]
It is readily seen from this identity is that every convex form is nonnegative; i.e., for every $d \in \N$, $\CONVEX{n}{2d} \subseteq \PSD{n}{2d}.$

\subsection{Tensors and outer product}\label{sec:tensors}
  A tensor of order $k$ is a multilinear form $T: (\mathbb R^n)^k \rightarrow \mathbb R$. The tensor $T$ is called symmetric if $T(\ix_1, \dots, \ix_k) = T(\ix_{i_1}, \dots, \ix_{i_k})$ for every $\ix_1, \dots, \ix_k \in \mathbb R^n$ and every permutation $(i_1,\dots, i_k)$ of the set $\{1, \dots, k\}$.
  The outer product of two vectors $\ix$ and $\iy$ is denoted by $\ix\otimes \iy$. The symmetric outer product $\frac12 (\ix \otimes \iy +\iy \otimes \ix)$ of two vectors $\ix$ and $\iy$ is denoted by $\ix \cdot \iy$. The (symmetric) outer product of a vector $\ix$ with itself $k$ times is denoted by $\ix^{k}$.
For any tensor $T$ of order $d$, the quantity $T(\ix_1, \dots, \ix_k)$ is a linear function of the outer product $\ix_1\otimes \dots \otimes \ix_k$ of the vectors $\ix_1, \dots, \ix_k$. If
the tensor $T$ is assumed to be symmetric, then this quantity only
depends on the symmetric outer product $\ix_1 \dots  \ix_k$.


\subsection{Forms and symmetric tensors} 
  For every form $p \in \FORMS{n}{k}$, there exists a unique symmetric tensor $T_p$ of order $k$ such that
\[p(\ix) = T_p(\underbrace{\ix, \dots, \ix}_{k \text{ times}}) \quad \forall \ix \in \mathbb R^n.\]
This is known as the {\it polarization} identity \cite{ehrenborg_apolarity_1993}. The tensor $T_p$ is related to the derivatives of the form $p$ via the relation
\begin{equation}
k! \; T_p(\ix_1, \dots, \ix_k) = \partial_{\ix_1}\dots\partial_{\ix_k}p \quad \forall \ix_1, \dots, \ix_k \in \mathbb R^n,\label{eq:polarization}
\end{equation}
and is related to the coefficients of the form $p$ via the identity
\begin{equation}
p_{i_1, \dots, i_n} = {k \choose i_1, \dots, i_k} \; T_p(\underbrace{\ie1, \dots, \ie1}_{\text{$i_1$ times}}, \dots, \underbrace{\ie n, \dots, \ie n}_{\text{$i_n$ times}}),
\end{equation}
where $p_{i_1, \dots, i_n}$ is the coefficient multiplying the
monomial $x_1^{i_1} \dots x_n^{i_n}$ in $p$.

When $k \eqqcolon 2d$ is even, we define the polynomial
  \begin{equation}
    Q_p(\ix, \iy) \coloneqq T_p(\underbrace{\ix, \dots, \ix}_{d \text{ times}}, \underbrace{\iy, \dots, \iy}_{d \text{ times}}) \quad \forall \ix,\iy \in \mathbb R^n.\label{eq:biform}
  \end{equation}
  We call the polynomial $Q_p$ the {\it biform} associated to $p$. We note that $Q_p$ is a form of degree $2d$ in the $2n$ variables $(\ix, \iy)$ that is homogeneous of degree $d$ in $\ix$ (resp. $\iy$) when $\iy$ (resp. $\ix$) is fixed.
  
\subsection{Inner product on $\FORMS{n}{2d}$\label{sec:inner_product}}
We equip the vector space $\FORMS{n}{2d}$ with the the following inner product
  \[\langle p, q \rangle \coloneqq p(\partial_{\ie 1}, \dots,
    \partial_{\ie n}) q \quad \forall p, q \in \FORMS{n}{2d},\]
  the so-called {\it Fischer} inner product
  \cite{diff_innner_product_fischer_1918}. This inner product can also be expressed in a more symmetric way in terms of the coefficients of the forms $p$ and $q$ as follows
  \[\langle p, q \rangle = (2d)! \sum_{i_1 + \dots + i_n = 2d} {2d \choose i_1, \dots, i_n}^{-1} p_{i_1, \dots, i_n} \; q_{i_1, \dots, i_n} \quad \forall p, q \in \FORMS{n}{2d}.\]
  By Riesz representation theorem, for every linear form $\ell: \FORMS{n}{2d} \rightarrow \R$, there exists a unique form $p \in \FORMS{n}{2d}$ satisfying
  \[\ell(q) = p(\partial_{\ie 1}, \dots, \partial_{\ie n})q \quad \forall q \in \FORMS{n}{2d},\]
  and we write $\ell = p(\partial_{\ie 1}, \dots, \partial_{\ie n})$.

  A particularly important special case of linear forms is given by tensor evaluations. The linear form given by $p \mapsto T_p(\ix_1, \dots, \ix_{2d})$ for some fixed vectors $\ix_1, \dots, \ix_{2d}$ is identified with the polynomial $\frac 1{(2d)!} \partial_{\ix_1} \dots \partial_{\ix_{2d}} $. For instance,
  \begin{itemize}
  \item The point evaluation map at $\ix \in \C^n$ given by $p \mapsto p(\ix)$ is equal to the differential operator $\frac1{(2d)!}\partial_\ix^{2d}$.
  \item The map
    $p \mapsto Q_p(\ix, \iy)$ is equal to the differential operator $\frac1{(2d)!} \partial_{\ix}^d \partial_{\iy}^d$ for all vectors $\ix$ and $\iy$ in $\C^n$.
  \item The map $p \mapsto Q_p(\iz, \bar \iz)$ is equal to the differential operator $\frac{1}{(2d)!} (\partial_{\ix}^2 + \partial_{\iy}^2)^d$ for any vector $\iz$ in $\C^n$ whose real and imaginary parts are given by $\ix$ and $\iy$. This follows from the fact that $\partial_{\iz} = \partial_{\ix} + i \partial_{\iy}$ and $\partial_{\bar \iz} = \partial_{\ix} - i \partial_{\iy}$.
  \end{itemize}

  
\subsection{Convex duality}
  We denote the dual of a convex cone $\Omega \subseteq \FORMS{n}{2d}$ by 
\[\Omega^* \coloneqq \{ \ell: \FORMS{2}{2d} \rightarrow \mathbb R \; | \; \ell \text{ is linear and } \ell(p) \ge 0 \; \forall p \in \Omega\}.\]
Recall that
$\CONVEX{n}{2d}^* = \cone\{ \ell_{\ix, \iy} \; | \; \ix, \iy \in
\R^n\},$ where
$\ell_{\ix, \iy}(p) \coloneqq \iy^T \nabla^2p(\ix)\iy$ and
$\cone(S)$ denotes the conic hull of a set $S$
\cite{reznick_blenders_2011}.  By using the pairing between linear
forms acting on the vector space $\FORMS{n}{2d}$ and elements of this
vector space described in \cref{sec:inner_product}, we can write
$\CONVEX{n}{2d}^* = \cone \{ (\partial_\iy)^2 (\partial_\ix)^{2d-2} \;
| \; \ix, \iy \in \R^n \}$.  For example, when $n=2$, if we denote
$\partial_{\ie 1}$ and $\partial_{\ie 2}$ by $\partial_x$ and
$\partial_y$ respectively, then
\begin{equation}\label{eq:dual_bivariate_convex}
\CONVEX{2}{2d}^* = \left\{ \sum_{k=1}^N (\alpha_k \partial_x+\beta_k \partial_y)^2 (\gamma_k \partial_x+\delta_k \partial_y)^{2d-2} \; | \; N \in \N \text{ and } \alpha_k,\beta_k,\gamma_k,\delta_k \in \R\right\}.
\end{equation}

\section{Generalized Cauchy-Schwarz Inequalities for Convex Forms}\label{sec:generalized_cauchy_schwarz}

The Cauchy-Schwarz inequality states that for any $n \times n$ positive semidefinite matrix $Q$,
\[\ix^TQ\iy \le \sqrt{\ix^TQ\ix  \cdot \iy^TQ\iy} \quad \forall \ix, \iy \in \mathbb R^n.\]
When the vectors $\ix$ and $\iy$ are complex and conjugate of each other, i.e. when $\ix = \bar \iy \eqqcolon \iz$, the inequality reverses as follows:
\[\sqrt{\iz^TQ\iz  \cdot \bar \iz^TQ \bar \iz}  \le \iz^TQ\bar\iz \quad \forall \iz \in \mathbb C^n.\]
This inequality is well-defined since the quantity appearing on the left-hand side is a nonnegative number as $\iz^TQ\iz \cdot \bar \iz^TQ \bar \iz = |\iz^TQ\iz|^2$, and the complex number on the right-hand side is a real number because it is equal to its conjugate.

The condition that the matrix $Q$ is positive semidefnite can be restated equivalently in terms of convexity of the quadratic form $p(\ix) \coloneqq \ix^TQ\ix$. %
In the following theorem, we present a generalization of these inequalities for convex forms of higher degree.



\begin{theorem}[Generalized Cauchy-Schwarz inequalities (GCS)]\label{thm:generalized_cauchy_schwarz}
  For any convex form $p$ in $n$ variables and of degree $2d$, we have
\begin{equation}\label{eq:real_generalized_cauchy_schwarz}
  Q_p(\ix, \iy) \le A_d \;\sqrt{p(\ix) p(\iy)} \quad \forall \ix, \iy \in \mathbb R^n,
\end{equation}
and
\begin{equation}\label{eq:complex_generalized_cauchy_schwarz}
  |p(\iz)| \le B_d \; Q_p(\iz, \bar \iz) \quad \forall \iz \in \mathbb C^n,
\end{equation}
where $Q_p$ is the biform associated with $p$ and defined in \cref{eq:biform}, and $A_d$ and $B_d$ are universal constants depending only on the degree $2d$.
\end{theorem}
The two remarks below give new interpretations of the GCS inequalities that do no involve the biform $Q_p$.
\begin{remark}
In view of the identification of differential operators with linear forms discussed in \cref{sec:inner_product}, the generalized Cauchy-Schwarz inequality in \cref{eq:real_generalized_cauchy_schwarz} can be written in terms of mixed derivatives as follows:
\[ \partial_\ix^d \partial_\iy^d p \le  A_d \;\sqrt{\partial_{\ix}^{2d}p \cdot \partial_{\iy}^{2d}p} \quad \forall \ix, \iy \in \R^n, \]
where $p$ is any convex form of degree $2d$.
Similarly, the second Generalized Cauchy-Schwarz inequality in \cref{eq:complex_generalized_cauchy_schwarz} can be written as 
\[|\partial_{\iz}^{2d}p| \le B_d \; \partial_{\iz}^d \partial_{\bar \iz}^d \, p \quad \forall \iz \in \C^n,\]
for any convex form $p$ of degree $2d$. If we denote by $\ix$ and $\iy$ the real and imaginary parts of the vector $\iz$, the same inequality reads
\[|\partial_{\iz}^{2d}p| \le B_d \; (\partial_{\ix}^2 + \partial_{\iy}^2)^d \, p.\]
\end{remark}
\begin{remark}
  For all forms $p \in \FORMS{n}{2d}$, and for all complex vectors $\iz \in \mathbb C^n$ whose real and imaginary parts are given by $\ix$ and $\iy$, the quantity $Q_p(\iz, \bar \iz)$ is proportional to the average of the form $p$ on the ellipse
  \[\{\alpha \ix + \beta \iy \; | \; \alpha, \beta \in \R \text{ and } \alpha^2 + \beta^2 \le 1\}.\]
  More precisely, we  show in \cref{sec:Qp_integral} the identity
  \begin{equation}
Q_p(\iz, \bar \iz) = \frac{4^d(d+1)}\pi {2d \choose d}^{-1} \; \iint_{\alpha^2+\beta^2 \le 1} p(\alpha \ix + \beta \iy) \; {\rm d}\alpha {\rm d}\beta.\label{eq:Qp_integral}
\end{equation}
The generalized Cauchy-Schwarz inequality in \eqref{eq:complex_generalized_cauchy_schwarz} can thus be equivalently written as
\[|p(\ix + i\iy)| \le B_d' \; \iint_{\alpha^2+\beta^2 \le 1} p(\alpha \ix + \beta \iy) \; {\rm d}\alpha {\rm d}\beta \quad \forall \ix, \iy \in \mathbb R^n,\]
for any convex form $p \in \FORMS{n}{2d}$, where $B_d' \coloneqq  \frac{4^d(d+1)}{{2d \choose d}\pi} B_d$.
\end{remark}

We emphasize that the constants $A_d$ and $B_d$ appearing in the GCS inequalities depend only on the degree $2d$, and not on the number of variables $n$. Furthermore, for the purposes of this paper, we need to find the smallest constants that make these GCS inequalities hold for quartic and sextic forms, i.e., when $d=2$ and $d=3$. This motivates the following definitions for all $d \in \N$:
\begin{equation}
\begin{aligned}
  A_d^* &\coloneqq \inf_{A \ge 0} \; A \quad \text{ s.t. }  Q_p(\ix, \iy) \le A \sqrt{p(\ix) p(\iy)} & \forall n \in \N,\; \forall p \in \CONVEX{n}{2d}, \; \forall \ix, \iy \in \mathbb R^n,\\
  B_d^* &\coloneqq \inf_{B \ge 0} B \quad \text{ s.t. } |p(\iz)| \le B  Q_p(\iz, \bar \iz) & \forall n \in \N, \; \forall p \in \CONVEX{n}{2d}, \; \forall \iz \in \mathbb C^n.
\end{aligned}
\label{eq:Ad_Bd_star}
\end{equation}

It should be clear that the ``$\inf$'' in these definitions is actually a ``$\min$'' since the inequality symbol ``$\le$'' appearing in the GCS inequalities is not strict. Moreover, the constants $A_d^*$ and $B_d^*$ are bounded below by 1 for all $d \in \N$. This is easily seen by, e.g., taking $n=1$, $\ix = \iy = \iz = 1$, and considering the (univariate) convex form $p(x) \coloneqq x^{2d}$. 

\cref{thm:generalized_cauchy_schwarz} is equivalent to the statement that the constants $A_d*$ and $B_d^*$ are finite for all positive integers $d$. The following theorem strengthens this claim.
\begin{theorem}[Optimal constants in the GCS inequalities]\label{thm:optimal_cs_constants}
For all positive integers $d$,
\[B_d^* = \frac{{2(d-1) \choose d-1}}{d}.\]
Moreover, $A_1^* = A_2^* = A_3^* = 1$,  $A_4^*$ is an algebraic number of degree 3, and for all even integers $d \ge 4$, $A_d^* > 1$. 
More generally, for every positive integer $d$, the constant $A_d^*$ is the optimal value of an (explicit) semidefinite program.
\end{theorem}
\begin{remark}
   The quantity $\frac{{2(d-1) \choose d-1}}{d}$ is known as the {\it $d^{\text{th}}$ Catalan number} \cite{catalan_hilton_1991}.
 \end{remark}

 The proofs of \cref{thm:generalized_cauchy_schwarz} and \cref{thm:optimal_cs_constants} are given in \cref{sec:proof_GCS} and \cref{sec:valueof_Ad_Bd} respectively.
 
 \subsection{Proof of the generalized Cauchy-Schwarz inequalities in \cref{thm:generalized_cauchy_schwarz}}
 \label{sec:proof_GCS}
 In this section, we will show that the GCS inequalities are, at heart, linear inequalities about bivariate convex forms. This observation will eventually lead to a simple proof of \cref{thm:generalized_cauchy_schwarz}.

 The next lemma leverages the homogeneity properties of the elements of $\FORMS{n}{2d}$ to linearize inequalities \cref{eq:real_generalized_cauchy_schwarz,eq:complex_generalized_cauchy_schwarz}.

 \begin{lemma}
   For all $n, d \in \N$, for any positive constants $A_d$ and $B_d$, and for any nonnegative form $p \in \PSD{n}{2d}$,\\
   (i) the form $p$ satisfies the inequality in \cref{eq:real_generalized_cauchy_schwarz} with constant $A_d$ if and only if
   \begin{equation}\label{eq:linear_real_generalized_cauchy_schwarz}
     2Q_p(\ix, \iy) \le A_d  \left(p(\ix) + p(\iy)\right) \quad \forall \ix, \iy \in \mathbb R^n,
   \end{equation}
   (ii) the form $p \in \PSD{n}{2d}$ satisfies the inequality in \cref{eq:complex_generalized_cauchy_schwarz} with constant $B_d$ if and only if
   \begin{equation}\label{eq:linear_complex_generalized_cauchy_schwarz}
     \RE{p(\iz)} \le B_d  Q_p(\iz, \bar \iz) \quad \forall \iz \in \mathbb C^n.
   \end{equation}
 \end{lemma}
 \begin{proof}
   Fix positive integers $n$ and $d$, positive scalars $A_d$ and $B_d$, and let $p \in \FORMS{n}{2d}$.
   Let us prove part (i) of the lemma first, i.e.,
   that the form $p$ satisfies \cref{eq:real_generalized_cauchy_schwarz} if and only if it satisfies \cref{eq:linear_real_generalized_cauchy_schwarz}. The ``only if'' direction can be easily seen from the inequality
   \[\sqrt{ab} \le \frac{a+b}2 \quad \forall a, b \ge 0.\]
   We now turn our attention to the ``if'' direction. Applying inequality \cref{eq:linear_real_generalized_cauchy_schwarz} to vectors $\ix$ and $\lambda^{\frac 1d} \iy$, where  $\lambda$ is a nonnegative scalar, results in
   \[2 Q_p(\ix, \lambda^{\frac 1{2d}} \iy) \le A_d \; \left(p(\ix) +  p(\lambda^{\frac 1{2d}} \iy)\right).\]
 By homogeneity, we get that
 $2\lambda Q_p(\ix, \iy) \le A_d \; \left(p(\ix) + \lambda^2 p(\iy)\right)$. In other words, the univariate polynomial $f(\lambda) \coloneqq A_d \;p(\ix) + \lambda^2\; A_d \;p(\iy) - 2 \lambda \; Q_p(\ix, \iy)$ is nonnegative on $[0, \infty)$. If one of the scalars $p(\ix)$ or $p(\iy)$ is zero, or if the scalar $Q_p(\ix, \iy)$ is negative, there is nothing to prove. Otherwise, the polynomial $f$ has two (complex) roots, whose sum and product are both positive. The polynomial $f$ is therefore nonnegative on $[0, \infty)$ if and only if its roots are equal or are not real, which happens if and only if the discriminant $Q_p(\ix, \iy)^2 - A_d^2 p(\ix)p(\iy)$ is nonpositive.

 We now prove part (ii) of the lemma, i.e., that the form $p$ satisfies \cref{eq:complex_generalized_cauchy_schwarz} if and only if it satisfies \cref{eq:linear_complex_generalized_cauchy_schwarz}. Again, it is straightforward to see why the ``only if'' part is true, so we only prove the ``if'' part. Let us assume that $p$ satisfies inequality \cref{eq:linear_complex_generalized_cauchy_schwarz} with constant $B_d$, and let $\iz$ be an arbitrary complex vector in $\mathbb C^n$. Let $\iz' = e^{i \theta}\iz$, where $\theta \coloneqq \frac{\arg(p(\iz))}{2d}$ is chosen so that $p(\iz')$ is a nonnegative scalar. By homogeneity, we have
  \[|p(\iz)| = \RE{p(\iz')} \text{ and } Q_p(\iz,\bar \iz) = Q_p(\iz', \bar \iz').\]
  Applying inequality \cref{eq:linear_complex_generalized_cauchy_schwarz} to $\iz'$ leads to $|p(\iz)| \le B_d \; Q_p(\iz, \bar \iz)$, which is the desired result.
\end{proof}

We now show that it suffices to prove the GCS inequalities for convex forms in 2 variables. For this purpose, notice that for any $n\text{-variate}$ form $p$ of degree $2d$ and for any two vectors $\ix, \iy \in \R^n$, the quantities $p(\ix), p(\iy)$ and $Q_p(\ix, \iy)$ appearing in inequality~\cref{eq:real_generalized_cauchy_schwarz} only depend on the form $p$ through its restriction to the plane spanned by the vectors $\ix$ and $\iy$ given by
\begin{equation}
q(x, y) \coloneqq p(x \ix + y \iy).\label{eq:restriction_p_plane}
\end{equation}
Indeed,
\[p(\ix) = q(\ie1), p(\iy) = q(\ie2) \text{ and } Q_p(\ix, \iy) = Q_q(\ie1, \ie2),\]
where $\ie1^T = (1, 0)$ and $\ie2^T = (0, 1)$. Moreover, for any complex vector $\iz = \ix + i \iy$, we have
$p(\iz) = q(\ie1 + i \ie2), Q_p(\iz, \bar \iz)= Q_q(\ie1 + i \ie2, \ie1 - i \ie2),$
and thus all the quantities appearing in the inequality \cref{eq:complex_generalized_cauchy_schwarz} only depend on $p$ through its two-dimensional restriction $q$ as well.

The form $q$ defined in \cref{eq:restriction_p_plane} is bivariate and of the same degree as $p$. Furthermore, the form $q$ is convex if $p$ is. The proof of theorem \cref{thm:generalized_cauchy_schwarz} therefore reduces to showing existence of two constants $A_d$ and $B_d$ indexed by $d \in \N$, such that all bivariate convex forms $q$ of degree $2d$ satisfy the inequalities
\begin{equation}\label{eq:bivariate_real_generalized_cauchy_schwarz}
  2Q_q(\ie1, \ie2) \le  \; A_d \; \left(q(\ie1) + q(\ie2)\right),
\end{equation}
and
\begin{equation}\label{eq:bivariate_complex_generalized_cauchy_schwarz}
  \RE{q(\ie1+i\ie2)} \le B_d \; Q_q(\ie1+i\ie2, \ie1-i\ie2).
\end{equation}
We now show that these inequalities follow from the following simple lemma, whose proof is delayed until the end of the section.

\begin{lemma}\label{lem:compact_cone}
  Let $\Omega$ be closed cone in $\FORMS{n}{2d}$, and let $\ell$ be a linear form defined on $\Omega$ that satisfies
  \begin{equation}
    \left[\ell(p) = 0 \implies p = 0, \; \forall p \in \Omega\right] \quad \text{ and } \quad \ell \ge 0\;\text{on}\; \Omega.\label{eq:condition_compactness}  
\end{equation}
   Then the set $\Omega_\ell \coloneqq \{ p \in \Omega \; | \; \ell(p) \le 1\}$ is compact.
\end{lemma}
\begin{proof}[Proof of \cref{thm:generalized_cauchy_schwarz}]
Fix $d \in \N$, and define two linear forms  $\ell$ and $s$ acting on $q \in \CONVEX{2}{2d}$ as
$\ell(q) \coloneqq q(\ie1) + q(\ie2)$ and $s(q) \coloneqq Q_q(\iz,\bar \iz)$, where $\iz = \ie1 + i\ie2$. We start by showing that the linear forms $\ell$ and $s$ satisfy the condition in \cref{eq:condition_compactness} with $\Omega = \CONVEX{2}{2d}$. If $q \in \CONVEX{2}{2d}$, then $q$ is nonnegative and therefore $\ell(q) \ge 0$. Moreover, because of the relationship between $Q_q(\iz, \bar \iz)$ and the integral of $q$ described in \cref{eq:Qp_integral}, it is clear that $s(q) \ge 0$ as well. Now assume that a form $q \in \CONVEX{2}{2d}$ satisfies $\ell(q) = 0$. Since $q$ is nonnegative, we have $q(\ie1) = q(\ie2) = 0$. By convexity, the restriction of the function $q$ to the segment linking $\ie1$ to $\ie2$ is identically zero. By homogeneity, $q$ must be identically zero. Similarly, if $s(q) = 0$, then by \cref{eq:Qp_integral}, the average of $q$ on the unit disk is zero, and since the form $q$ is assumed to be nonnegative, it must be identically 0.

By \cref{lem:compact_cone}, the following two sets must therefore be compact:
\[L \coloneqq \{q \in \CONVEX{2}{2d} \; | \; q(\ie1) + q(\ie2) \le 1\}, \quad S \coloneqq \{q \in \CONVEX{2}{2d} \; | \; Q_p(\ie1+i\ie2,\ie1-i\ie2)  \le 1\}.\]
Let $\|\cdot\|$ be any norm on the vector space $\FORMS{2}{2d}$ and define the scalars $\alpha, \beta$ as follows:
\[\alpha \coloneqq \sup_{q \in L} \|q\| \text{ and } \beta \coloneqq \sup_{q \in \FORMS{2}{2d}} \frac{Q_q(\ie1, \ie2)}{\|q\|}.\]
We will show that inequality \cref{eq:bivariate_real_generalized_cauchy_schwarz} holds with constant $A_d = \alpha\beta$.
Note that $\alpha$ is finite because $L$ is compact, and $\beta$ is finite because the map $q \mapsto Q_q(\ie1, \ie2)$ is a linear function over a finite dimensional space. Let $q \in \CONVEX{2}{2d}$ and assume that $q$ is not zero, so that the scalar $\ell(q)$ is positive. On the one hand, we have $Q_q(\ie2, \ie2) \le \beta \|q\|$. On the other hand, $\|q\| \le \alpha \ell(q)$ because the form $\frac{q}{\ell(q)}$ is in the set $L$. We have just shown that $Q_q(\ie2, \ie2) \le \alpha \beta \ell(q)$, which concludes the proof of inequality \eqref{eq:bivariate_real_generalized_cauchy_schwarz}. A similar argument shows the existence of a finite constant $B_d$ for which \eqref{eq:bivariate_complex_generalized_cauchy_schwarz} hold.
\end{proof}

\begin{proof}[Proof of \cref{lem:compact_cone}]
  Let $\Omega$ and $\ell$ be as in the statement of the lemma, and let $\| \cdot \|$ be any norm of $\FORMS{n}{2d}$.
  It is clear that the set $\Omega_\ell$ is a closed set as it is the intersection of a half space with $\Omega$.

  Suppose it is not bounded, i.e., suppose there exists a sequence $(q^{(k)})_k$ of $\Omega_\ell$ such that 
  $\|q^{(k)}\| \rightarrow \infty \text{ as } k \rightarrow \infty$.
  By taking a subsequence if necessary, we can assume that the sequence $(q^{(k)})_k$ does not contain zero.
  The sequence $\left(\frac{q^{(k)}}{\|q^{(k)}\|}\right)_k$ lives in the cone $\Omega$, and is bounded (by one), so it admits a converging subsequence. Let $q_\infty \in \Omega$ denote its limit.
  Since the function $\ell$ is bounded by $1$ on $\Omega_\ell$,  the ratio $\frac{\ell(q^{(k)})}{\|q^{(k)}\|}$ tends to $0$ as $k \rightarrow \infty$, and therefore $\ell(q_\infty) = 0$. By~\cref{eq:condition_compactness}, the form $q_\infty$ is itself identically zero. Yet, $\|q_\infty\| = 1$, which is a contradiction.
\end{proof}

\subsection{Values of the optimal constants $A_d^*$ and $B_d^*$ defined in \cref{eq:Ad_Bd_star}}
\label{sec:valueof_Ad_Bd}
Fix $d \in \N$. We have shown in the previous section that
\begin{equation}
  \begin{aligned}
A_d^* &= \min A \text{ s.t. } \; 2\; Q_q(\ie1+i\ie2, \ie1-i\ie2) \le A\left(q(\ie1) + q(\ie2)\right) \quad &\forall q \in \CONVEX{2}{2d},\\
B_d^* &= \min B \text{ s.t. } \; \RE{q(\ie1+i\ie2)} \le B  Q_q(\ie1+i\ie2, \ie1-i\ie2) \quad &\forall q \in \CONVEX{2}{2d}.
  \end{aligned}
\label{eq:primal}
\end{equation}
This formulation is useful for finding lower bounds on the constants $A_d^*$ and $B_d^*$. Indeed, to show that $A_d^* > A$ for some scalar $A$ , it suffices to exhibit a convex bivariate form $q$ that satisfies $2Q_q(\ie1+i\ie2, \ie1-i\ie2) > A\left(q(\ie1) + q(\ie2)\right)$. An similar statement can be made for $B_d^*$ as well.

To find upper bounds on the constants $A_d^*$ and $B_d^*$, we take a dual approach. For all scalars $A$ and $B$, we define the linear forms
\begin{equation}
  \ell_A \coloneqq A (\partial_x^{2d}+\partial_y^{2d}) - 2 \partial_x^d\partial_y^d,
  \label{eq:ell_A}
\end{equation}
and
\begin{equation}
  s_B \coloneqq B(\partial_x^2+\partial_y^2)^{d} - \RE{(\partial_x + i \partial_y)^{2d}}.
  \label{eq:s_B}
\end{equation}
Because of our discussion in \cref{sec:inner_product}, the constants $A_d^*$ and $B_d^*$ can be found by solving the following optimization problems, dual to the optimization problems in \cref{eq:primal}.
\begin{equation}
  \begin{aligned}
  A_d^* &= \min A \; &\text{ s.t. } \quad&\ell_A \in \CONVEX{2}{2d}^*,\\
  B_d^* &= \min B     &\text{ s.t. } \quad&s_B \in \CONVEX{2}{2d}^*.
\end{aligned}
  \label{eq:dual}
\end{equation}
In other words, in order to prove that $A_d^* \le A$ for some scalar $A$, one has to show that $\ell_A$ can be decomposed as in \cref{eq:dual_bivariate_convex}. An identical statement can be made for $B_d^*$ here too.

\subsubsection{Values of the optimal constants $A_d^*$ defined in \cref{eq:Ad_Bd_star}}
We will show in this section that the optimization problems in \cref{eq:dual} are tractable. The following theorem shows that convex bivariate forms are also sos-convex.

\begin{theorem}\cite[Theorem 5.1]{ahmadi2013complete}\label{thm:cd_is_spectrahedron}
  A bivariate form $q(x,y) = \sum_{i=0}^{2d}q_ix^iy^{2d-i}$ is convex if and only if it is sos-convex, i.e., if there exists a positive semidefinite $2d \times 2d$ matrix $Q$ such that
      \begin{equation}\label{eq:convex_hessian_sos}
      \iu^T\nabla^2q(x,y)\iu = z^TQz \quad \forall x,y \in \R, \forall \iu \in \R^2,
      \end{equation}
      where
      $z^T \coloneqq (u_1 x^{d-1}, u_1x^{d-2}y, \dots, u_1y^{d-1}, u_2 x^{d-1}, u_2x^{d-2}y, \dots, u_2y^{d-1})$
      is the vector of monomials in the variables $x,y,u_1,u_2$.
\end{theorem}
By expanding both sides of \cref{eq:convex_hessian_sos} and matching the coefficients of the polynomials that appear on both sides, we obtain an equivalent system of linear equations involving the coefficients of the form $q$ and the entries of the matrix $Q$. What we have just shown is that
\[\CONVEX{2}{2d} = \{ q \in \FORMS{2}{2d} \; | \; \exists Q \succeq 0 \text{ s.t. $q$ and $Q$ satify the linear equations in \cref{eq:convex_hessian_sos}}\}.\]

This set is a {\it projected spectrahedron}, i.e., it is defined via linear equations and linear matrix inequalities. The class of projected spectrahedra is stable by taking the convex dual, so $\CONVEX{2}{2d}^*$ is also a projected spectrahedron. Optimizing linear functions over such sets (or their duals) is therefore an semidefinite program (SDP). Semidefinite programming is a well-studied subclass of convex optimization problems that can be solved to arbitrary accuracy in polynomial time. 
Because of our characterization of $A_d^*$ in \cref{eq:dual}, it should be clear that $A_d^*$ is the optimal value of an SDP.
We report in \cref{tbl:solution_sdp} the values of $A_d$ to $4$ digits of accuracy obtained using the solver MOSEK~\cite{mosek}.

\begin{table}\centering
  \begin{tabular}{|c|l|l|l|l|l|l|l|l|}
    \hline
    $d$   & 1 & 2 & 3 & 4 & 5 & 6 & 7 & 8 \\
    \hline
    $A_d^*$ & 1.000  & 1.000  & 1.000  & 1.011 & 1.000 & 1.061 & 1.000 & 1.048 \\
    \hline
  \end{tabular}
  \caption{\label{tbl:solution_sdp}Approximation of the value of the constant $A_d^*$ defined in \cref{eq:Ad_Bd_star} obtained by numerically solving the SDP in \cref{eq:dual}}
\end{table}

Note that in practice, numerical software will only return an {\it approximation} of the optimal solution to an SDP. Such approximations can nevertheless be useful as they help formulate a ``guess'' for what the exact solution might be, especially if the solution sought contains only rational numbers (with small denominators). In particular, the following identities, which are trivial to verify, were obtained by rounding solutions obtained from a numerical SDP solver, and constitute a formal proof that $A_1^* = A_2^* = A_3^* = 1$. 
\begin{align*}
  \partial_x^2+\partial_y^2 - 2\partial_x\partial_y &= (\partial_x-\partial_y)^2,\\
  \partial_x^4+\partial_y^4 - 2 \partial_x^2\partial_y^2 &= (\partial_x-\partial_y)^2 (\partial_x+\partial_y)^2,\\
  \partial_x^6+\partial_y^6 - 2 \partial_x^3\partial_y^3 &= \frac12 (\partial_x-\partial_y)^2 ( \partial_x^4 + \partial_y^4 + (\partial_x+\partial_y)^4).
\end{align*}

We also note that few algebraic methods have been developed for solving SDPs exactly, especially for problems of small sizes \cite{parrilo_exact_sdp_2003, henrion_exact_sdp_2016}. For instance, we were able to solve the SDP in \eqref{eq:dual}  characterizing $A_d^*$ for $d=4$. The key steps in this computation are (i) exploiting the symmetries of the problem to reduce the size of the SDP \cite{parrilo_symmetry_2004}, (ii) formulating the corresponding {\it Karush-Kuhn-Tucker} equations (see, e.g., \cite{lasserre_handbook_sdps_2011}) and (iii) solving these polynomial equations using variable elimination techniques.\footnote{The curious reader is referred to \href{http://www.bachirelkhadir.com/research/code/Exact_value_of_A4_star.ipynb}{this Sage notebook} describing these steps in more details  \cite{sage_2010}.}
The value of $A_4^*$ is given by
\[\frac{1}{70} \, \omega^{\frac 13} + \frac{128}{15} \, \omega^{-\frac13} + \frac{11}{35}, \text{ where } \omega \coloneqq 14336 + i \frac{14336\sqrt{3}}{9}.\]
Unlike the constants $A_d^*$ for $d \le 3$, the constant $A_4^*$ is not equal to 1. In fact $A_4^*$ is not even a rational number, but an algebraic number of degree three with minimal polynomial given by
\[t^3 - \frac{33}{35}\;t^2 - \frac{17}{245}\;t + \frac{13}{42875}.\]
In \cref{sec:A_d_even}, we prove that $A_d^* > 1$ whenever $d$ is an even integer larger than $4$.

{\bf Conjecture and open problem.} Supported by the numerical evidence in \cref{tbl:solution_sdp}, we conjecture that the constant $A_d^*$ defined in \eqref{eq:Ad_Bd_star} is equal to 1 when the integer $d$ is odd, and we leave open the problem of finding the exact value of $A_d^*$ for even integers $d$ larger than $4$.

\subsubsection{Exact values of the optimal constants  $B_d^*$ defined in \cref{eq:Ad_Bd_star}}
The goal of this section is to show that
\[B_d^* = \frac{{2(d-1)\choose d-1}}{d} \quad \forall d \in \N.\]

%
%
%
\noindent The following proposition shows that for all positive integers $d$, $s_B \in \CONVEX{2}{2d}^*$ for $B = \frac{{2(d-1)\choose d-1}}{d}$, where $s_B$ is defined in \cref{eq:s_B}, and therefore, $B_d^* \le \frac{{2(d-1)\choose d-1}}{d}$.

\begin{proposition}\label{prop:proof_complex_cs}
  For all positive integers $d$, for all $x,y \in \R$,
  \begin{equation}\label{eq:proof_complex_cs}
  \frac{{2(d-1)\choose d-1}}{d} (x^2+y^2)^d - \RE{(x+i y)^{2d}} = \frac{4^d}{2d} \sum_{k=0}^{d-1} \left(-s_kx + c_k y\right)^2 \; \left( c_k x + s_k y\right)^{2d-2},
\end{equation}
where $c_k = \cos(\frac{k\pi}{2d})$ and $s_k = \sin(\frac{k\pi}{2d})$
for $k = 0, 1, \dots, 2d-1$.
\end{proposition}

\begin{proof}
  Identity \cref{eq:proof_complex_cs} is homogeneous in $\ix^T = (x, y)$. It is therefore sufficient to prove that it holds when $\ix$ is a unit vector. Let $\ix$ be such a vector, and let us write $x = \cos(\theta)$ and $y = \sin(\theta)$ for some $\theta \in \R$. Then identity \cref{eq:proof_complex_cs} becomes
  \[\frac{{2(d-1)\choose d-1}}{d} - \cos(2d \theta) = \frac{4^d}{2d} \sum_{k=0}^{d-1} \sin^2\left(\frac{k\pi}{2d}-\theta\right) \cos^{2d-2}\left(\frac{k\pi}{2d}-\theta\right) \; \forall \theta \in \R.\]
  The proof of this trigonometric identity is purely computational, and can be found in \cref{app:trigo}.
\end{proof}

Let us now show that for all  $d \in \N$, the constant $B_d^*$ is bounded below by  $\frac{{2d-2 \choose d-1}}d$. To do so, we exhibit  a family of nonzero bivariate convex forms $(q_d)_{d \in \N}$ that satisfy
\[ \RE{q_d(\ie1+i\ie2)} =  \frac{{2d-2 \choose d-1}}d \; Q_{q_d}(\ie1+i\ie2,\ie1-i\ie2) \quad \forall d \in \N .\]
We plot in \cref{fig:optimal_poly_complex_cs} the $1\text{-level}$ sets of the polynomials $q_d$ for $d=1,\dots,4$.
  
\begin{proposition}\label{prop:optimal_poly_complex_cs}
  For every positive integer $d$, the form $q_d$ defined by
  \begin{equation}
  q_d(x, y) \coloneqq \RE{ (x+i \; y)^{2d} } +  (2d-1)  (x^2+y^2)^d\label{eq:optimal_poly_complex_cs}
\end{equation}
  is convex and satisfies $\RE{q_d(\ie1+i\ie2)} =  \frac{{2d-2 \choose d-1}}d \; Q_{q_d}(\ie1+i\ie2,\ie1-i\ie2)$.
\end{proposition}

\newcommand{\LabelSet}[1]
  {\begin{tikzpicture}%
  \draw[thick,->,>=latex] (-1.1,0)--(1.1,0) node[above] {$x$};%
  \draw[thick,->,>=latex] (0,-1.1)--(0,1.1) node[left] {$y$};%
    \draw[dashed,domain=0:2*pi,scale=1,samples=100] plot %
    ({deg(\x)}:{1 / (2*#1-1+cos(\x * 2 * #1 r))^(1/2/#1)});%
  \end{tikzpicture}}%

\begin{figure}[h!]
  \centering
  \begin{subfigure}[t]{.4\linewidth}
    \centering
    \begin{tikzpicture}%
      \draw[thick,->,>=latex] (-1.1,0)--(1.1,0) node[above] {$x$};%
      \draw[thick,->,>=latex] (0,-1.1)--(0,1.1) node[left] {$y$};%
      \draw[dashed] (1/1.41421,-1.1)--(1/1.41421,1.1);
      \draw[dashed] (1/-1.41421,-1.1)--(1/-1.41421,1.1);
    \end{tikzpicture}\\
    $q_1(x, y) = 2x^2$
  \end{subfigure}
  \begin{subfigure}[t]{.4\linewidth}
    \centering
    \LabelSet{2}\\
    $q_2(x, y) =4x^{4} + 4y^{4}$
  \end{subfigure}
  \begin{subfigure}[t]{.4\linewidth}
    \centering
    \LabelSet{3}\\
    $q_3(x, y) =6x^{6}+30x^{2} y^{4}+4y^{6}$
  \end{subfigure}
  \begin{subfigure}[t]{.4\linewidth}
    \centering
    \LabelSet{4}\\
    $q_4(x, y) = 8x^{8}+ 112x^4y^4 + 8y^{8}$
  \end{subfigure}
  \caption{Plot of the $1\text{-level}$ sets of the forms $q_d$ defined in \eqref{eq:optimal_poly_complex_cs} for $d=1,\dots,4$. These forms saturate the generalized Cauchy-Schwarz inequality in \eqref{eq:complex_generalized_cauchy_schwarz} and\label{fig:optimal_poly_complex_cs}}
\end{figure}
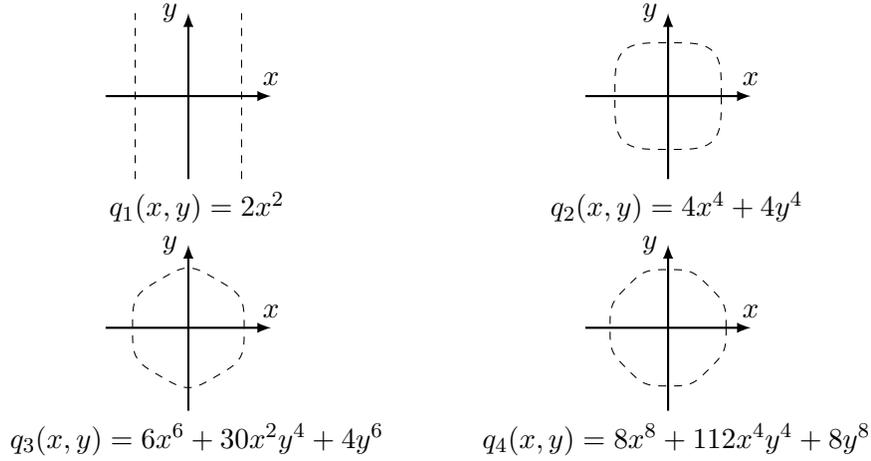

To give the proof of \cref{prop:optimal_poly_complex_cs}, it will be convenient for us to switch to {\it polar coordinates} $(r,\theta) \in [0, \infty) \times [0, 2\pi)$ defined by$x = r \cos(\theta) \text{ and } y = r \sin(\theta).$
More explicitly, for every $k \in \N$, every bivariate form $p \in \FORMS{2}{k}$ can be expressed in polar coordinates as $p(x, y) = r^{k} f(\theta)$, where $f$ is a polynomial expression in $\cos(\theta)$ and $\sin(\theta)$. In particular, the function $f$ is differentiable infinitely many times. The following lemma gives the expressions of the Hessian and Laplacian operators in polar coordinates. 
  
\begin{lemma}[Hessian and Laplacian in polar coordinates]
  \label{lem:diff_in_polar_coordinates}
    The Hessian and Laplacian of a form $p \in \FORMS{2}{k}$, whose expression in polar coordinate is $p(x,y) = r^{k}f(\theta)$, are given by 
    \[\nabla^2p(x,y) = r^{k-2} \; \left( k(k-1) f(\theta) \ie{rr} + (k-1) f'(\theta) \ie {r\theta} + \left(k +  f''(\theta) \right) \ie {\theta \theta }\right),\]
      \[\Delta p(x, y) = r^{k-2}\; \left( k^2 f(\theta) + f''(\theta)\right),\]
    where $\ie{r} \coloneqq \begin{pmatrix}\cos(\theta)\\\sin(\theta)\end{pmatrix}, \ie{\theta}\coloneqq \begin{pmatrix}-\sin(\theta)\\\cos(\theta)\end{pmatrix},$
    and $\ie{rr} = {\ie r} {\ie r}^T, \ie{r\theta} = {\ie r} {\ie\theta}^T +  {\ie\theta}{\ie r}^T, \ie{\theta\theta} = {\ie \theta} {\ie\theta}^T$.
\end{lemma}
\begin{proof}See \cref{sec:proof_hessian_laplacian_polar}.\end{proof}

\begin{proof}[Proof of \cref{prop:optimal_poly_complex_cs}]
  
  Fix a positive integer $d$ and let us prove that the form $q_d$ defined in \cref{eq:optimal_poly_complex_cs} is convex. Note that we can express $q_d$ in polar coordinates as follows,
  \[q_d(x, y) = \RE{r^{2d} e^{i2d\theta} + (2d-1) r^{2d}}.\]
Using \cref{lem:diff_in_polar_coordinates}, we get that
\[\nabla^2 r^{2d} =   r^{2d-2} \left(2d(2d-1) \ie {rr} + 2d \ie{\theta \theta} \right)\]
and
\[\nabla^2 \left(r^{2d}e^{i 2d\theta}\right) =  2d(2d-1) r^{2d-2}e^{i 2d\theta} \left(\ie{rr} + i  \ie{r \theta} - \ie {\theta \theta}\right).\]
By summing the previous two equations term by term and taking the real part, we get
\[\nabla^2 q_d(x,y) =  2d(2d-1)r^{2d-2} \begin{pmatrix} \cos(2d\theta) + 2d-1  & - \sin(2d\theta) \\- \sin(2d\theta) & - \cos(2d\theta) + 1\end{pmatrix}.\]
  The trace of the matrix in the right-hand side of this equation is $(2d)^2(2d-1)r^{2d-2}$, and its determinant is given by $(2d(2d-1)r^{2d-2})^2(2d-2)(1 + \cos(2d\theta) )$. Both the trace and the determinant of the Hessian matrix of $q_d$ are thus clearly nonnegative. This proves that this Hessian matrix is positive semidefinite and that the form $q_d$ is convex.

  Let us now compute $\RE{q_d(\iz)}$ and $Q_{q_d}(\iz, \bar \iz) \text{ where } \iz = \ie1+i\ie2$. By plugging $x = 1$ and $y = \iim$ in the right-hand side of the identity
\[\RE{(x+iy)^{2d}} = \sum_{k=0}^d  {2d \choose 2k} x^{2d-2k} (iy)^{2k},\]
we get that $q_d(\iz) = \sum_{k=0}^d  {2d \choose 2k} = 2^{2d-1}$.

We now compute $Q_{q_d}(\iz, \bar \iz)$. Because of the identification between linear forms and differential operators introduced in \cref{sec:inner_product}, this task is equivalent to computing $\Delta^{d}q_d$.
On the one hand, the function $f(x, y) \coloneqq (x+i y)^{2d}$ is holomorphic when viewed as a function of the complex variable $z = x+\iim y$, therefore
\[\Delta^d \RE{(x+\iim y)^{2d}} = 0.\]
On the other hand, by using \cref{lem:diff_in_polar_coordinates} again, we get for every positive integer $k$, $\Delta r^{2k} = 4 k^2 r^{2k-2}$, and by immediate induction, $\Delta^d r^{2d} = 2^{2d} d!^2$.
Overall, we get $\Delta^d q_d = (2d-1) 2^{2d} d!^2$, and therefore
\[Q_{q_d}(\iz, \bar \iz) = 2^{2d-1} \frac{d}{ {2d-2 \choose d-1} }.\]
In conclusion, we have just proved that
$\RE{q_d(\iz)}  = \frac{ {2d-2 \choose d-1} }{d} Q_{q_d}(\iz, \bar \iz).$
\end{proof}

\section{What Separates the Sum of Squares Cone from the Nonnegative Cone}
\label{sec:sos_vs_psd}
In \cite{blekherman_nonnegative_2012}, the author offers a complete description of the hyperplanes separating sos forms from non-sos forms inside the cone of nonnegative quaternary quartics. We include the high level details of that description here to make this article relatively self-contained.

If a form $p \in \PSD{4}{4}$ is not sos, then there exists a subset $V = \{\iv_1, \dots, \iv_8\}$  of $\mathbb C^n$ and complex numbers  $a_1, \dots, a_8 \in \mathbb C\setminus\{0\}$ that certify that fact in the sense that
\begin{equation}\label{eq:sos_sepearating_hyperplane}
\sum_{i=1}^8 a_i q(\iv_i) \ge 0 \quad \forall q \in \Sigma_{4, 4},
\end{equation}
but $\sum_{i=1}^8 a_i p(\iv_i) < 0$ \cite[Theorem 1.2]{blekherman_nonnegative_2012}.
Let us now explain where the set $V$ and the scalar $a_i$ come from.
The points in $V$ are the common zeros to three linearly independent quadratic forms $q_i(\ix)~=~\ix^TQ_i\ix$, where the $Q_i$ are $4 \times 4$ symmetric matrices and $i=1,2,3$  \cite[Lemma 2.9]{blekherman_nonnegative_2012}. Equivalently,
\[V = \{ \ix \in \mathbb R^4 \; | \;  q_1(\ix) = q_2(\ix) = q_3(\ix) \} = \{\iv_1, \dots, \iv_8\}.\]
Define $V^2 \coloneqq \{ \iv\iv^T \; | \; \iv \in V\}$. The eight elements of $V^2$ live in the $6\text{-dimensional}$ vector space of symmetric $4 \times 4$ matrices, and they are all orthogonal to the three-dimensional vector space spanned by $Q_1,Q_2$, and $Q_3$.
A simple dimension counting argument tells us that there must exist a linear relationship between the vectors $\iv_i$ of the form
\begin{equation}
\sum_{i=1}^8 \mu_i \iv_i\iv_i^T = 0,\label{eq:cayley-bacharach}
\end{equation}
for some $\mu_1, \dots, \mu_8 \in \mathbb C$. In fact, this relationship between the $\iv_i$ is unique (up to scaling). Furthermore, all the scalars $\mu_i$ must be nonzero. This is known as the Cayley-Bacharach relation \cite{bacharach_eisenbud_1996}. We assume from now on that all the $\mu_i$ have norm 1 (after possibly scaling the vectors $\iv_i$.)

Now that we have characterized the evaluation points $\iv_i$, let us turn our attention to the scalars $a_i$ in \cref{eq:sos_sepearating_hyperplane}. These scalars should satisfy \cite[Theorem 6.1 and Theorem 7.1]{blekherman_nonnegative_2012}
\begin{equation}
\sum_{i=1}^8 \frac1{a_i} = 0.\label{eq:sum_reciprocals_a}
\end{equation}

We now need to distinguish between the case where all the elements of $V$ are real (i.e., $V \subset \mathbb R^4$) and the case where they are not. 
In the former case, all the scalars $\mu_i$ must be real, exactly one of the scalars $a_i$ must be negative and the rest should be positive \cite[Theorem 6.1]{blekherman_nonnegative_2012}. By reordering the indices if necessary, we assume $a_1 < 0$ and $a_i > 0$ for $i > 1$. By scaling  all the scalars $a_i$, we assume
\begin{equation}
\frac1{a_1} = -\sum_{i=2}^8 \frac1{a_i} = -1,\label{eq:reciprocals_real}
\end{equation}
in which case inequality \cref{eq:sos_sepearating_hyperplane} reads
\begin{equation}\label{eq:real_hyperplane}
p(\iv_1) \le \sum_{i=2}^8 a_i p(\iv_i).
\end{equation}



In the case where one of the vectors $\iv_i$ is not real, it is proven in \cite[Corollary 4.4]{blekherman_nonnegative_2012} that $V$ could be taken so that exactly two of the vectors $\iv_i$ are not real, in which case they (and their coefficients $\mu_i$) should be conjugate of each other. Again, up to reordering, we can assume that $\iv_1 \coloneqq \iz$ is not real, $\iv_2 = \bar \iz$, $\mu_1 = \bar \mu_2$ and the rest of the vectors $\iv_i$ and scalars $\mu_i$ are real. By scaling, we can further assume that
\begin{equation}
\frac1{a_1} + \frac{1}{\bar a_1} = -\sum_{i=2}^8 \frac1{a_i} = -1.\label{eq:reciprocals_complex}
\end{equation}
In this case, the inequality in \cref{eq:sos_sepearating_hyperplane} reads
\begin{equation}\label{eq:complex_hyperplane}
a_1p(\iz) + \bar a_1 p(\bar \iz) + \sum_{i=3}^8 a_i p(\iv_i) \ge 0.
\end{equation}

We present the following simple lemma (whose proof can be found in \cref{app:proof_lemma}) that will let us rewrite the inequality in \cref{eq:sos_sepearating_hyperplane} without refering to the scalars $a_i$.
\begin{lemma}\label{lem:holder}
  For all nonnegative scalars $x_2, \dots, x_n,$ the maximum of the quantity  $\sum_{i=2}^n a_ix_i$ over all positive scalars $a_2, \dots, a_n$ satisfying $\sum_{i=2}^n \frac1{a_i} = 1$ is $(\sum_{i=2}^n \sqrt{x_i})^2$.
  
Furthermore, for any complex number $z$, the maximum value of the quantity $az + \bar a \bar z$ over all complex numbers $a$ satisfying $\frac1a + \frac1{\bar a} = 1$ is $2(|z| + \RE{z}).$
\end{lemma}

Indeed, \cref{lem:holder} shows that a form $p$ satisfies inequality \cref{eq:real_hyperplane} for every $a_1, \dots, a_8 \in \R$ satisfying \cref{eq:reciprocals_real} if and only if
\[p(\iv_1) \le \left(\sum_{i=2}^8 \sqrt{p(\iv_i)}\right)^2,\]
and the same form satisfies Inequality \cref{eq:complex_hyperplane} for every $a_1 \in \C$ and $a_3 \dots, a_8 \in \R$ satisfying \cref{eq:reciprocals_complex} if and only if
\[2(|p(\iz)| + \RE{p(\iz)}) \le \left(\sum_{i=3}^8 \sqrt{p(\iv_i)}\right)^2.\]
We summarize the result of this section in the following theorem.

\begin{theorem}[\cite{blekherman_nonnegative_2012}]\label{thm:caracterization-sos}
  A nonnegative quaternary quartic form $p$ is sos if and only if both of the following conditions hold.
  \begin{itemize}
  \item For every $\iv_1, \dots, \iv_8 \in \mathbb R^4$ and $\alpha_2, \dots, \alpha_8 \in \{-1, 1\}$ such that $\iv_1\iv_1^T = \sum_{i=2}^8 \alpha_i \iv_i\iv_i^T$,
    \begin{equation}
    \label{eq:first_implication}
    p(\iv_1) \le \left(\sum_{i=2}^8  \sqrt{p(\iv_i)}\right)^2.
    \end{equation}
  \item For every $\iz \in \mathbb C^4$, for every $\iv_3, \dots, \iv_8 \in \mathbb R^4$, and for every $\alpha_3, \dots, \alpha_8 \in \{-1, 1\}$ such that $\iz\iz^T + \bar \iz \bar \iz^T = \sum_{i=3}^8 \alpha_i \iv_i\iv_i^T$,
    \begin{equation}
    \label{eq:second_implication}
     2(|p(\iz)| + \RE{p(\iz)}) \le \left(\sum_{i=3}^8 \sqrt{p(\iv_i)}\right)^2.
  \end{equation}
  \end{itemize}
\end{theorem}

\begin{figure}[h!]
  \centering
  \includegraphics[scale=0.1]{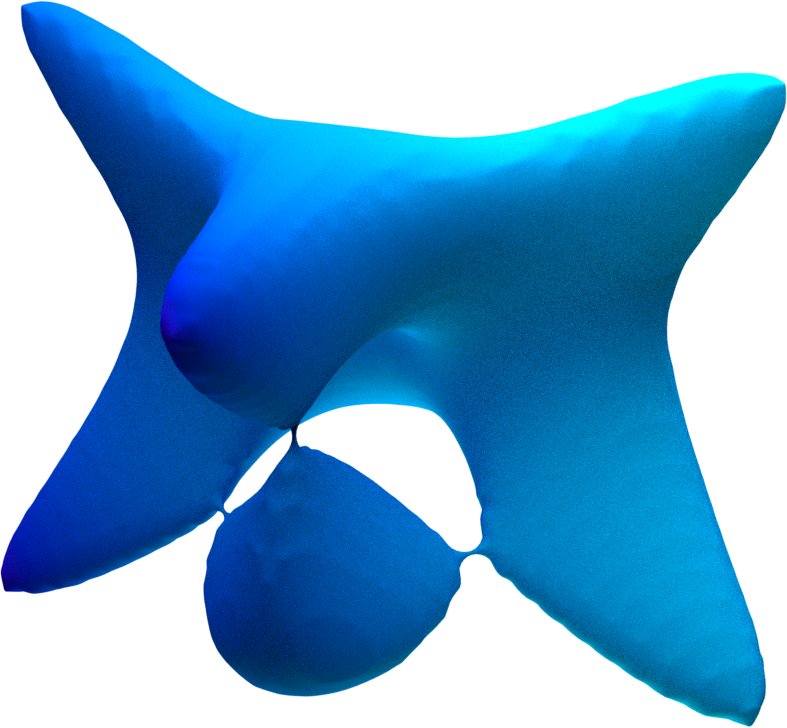}
  \caption{\label{fig:one_level_set}The $1\text{-level}$ set of the polynomial $p$ defined in \cref{eq:def-example-p}.}
\end{figure}

\begin{example}
  Let us use \cref{thm:caracterization-sos} to prove that the following quaternary quartic form  
  \begin{equation}\label{eq:def-example-p}
    p(\ix) = \sum_{i=1}^4 x_i^4 + \sum_{\substack{1 \le i,j,k \le 4\\i\ne j, i \ne k, j\ne k}} x_i^2x_jx_k + 4 x_1x_2x_3x_4,  
  \end{equation}
  whose $1\text{-level}$ set is plotted in \cref{fig:one_level_set}, is not sos. 
  Take $V$ to be the set of 8 elements given by
  \[V \coloneqq \{-1, 1\} \times \{-1, 1\} \times \{-1, 1\} \times \{1\},\]
  and partition it as   $V = V^+ \cup V^-$,  where $V^+$ (resp. $V^-$) is the subset of elements $V$ whose entries sum to an even (resp. odd) number. Up to scaling, the unique linear relationship satisfied by the elements of $V$ is given by
  \[\sum_{\iv \in V^+}  \iv\iv^T - \sum_{\iv \in V^-}  \iv\iv^T = 0.\]
  Let $\iv_1 \in V$ stand for the vector $(1, 1, 1, 1)^T$, and denote the rest of the elements of $V$ by $\iv_2, \dots, \iv_8$. It is easy to check that
  \[p(\iv_1) = 32 \text{ and } p(\iv_i) = 0 \text{ for } i=2,\dots,8,\]
  and therefore $p$ does not satisfy requirement \cref{eq:first_implication} in \cref{thm:caracterization-sos}, and as result, $p$ is not sos as a result.
\end{example}

\section{Proof of the Main Theorem}
\label{sec:proof_main_theorem}
In this section we prove that $\CONVEX{4}{4} \subseteq \SOS{4}{4}$. Our plan of action is to show that any quaternary quartic form $p$ that satisfies the following generalized inequality:
\begin{equation}
Q_p(\ix, \iy) \le \sqrt{p(\ix) p(\iy)} \quad \forall \ix, \iy \in \mathbb R^n\label{eq:cs-quartic-1}
\end{equation}
and
\begin{equation}
|p(\iz)| \le Q_p(\iz, \bar \iz) \quad \forall \iz \in \mathbb C^n,\label{eq:cs-quartic-2}
\end{equation}
must satisfy the requirements \cref{eq:first_implication} and \cref{eq:second_implication} that appear in \cref{thm:caracterization-sos}, and hence must be sos. The containment $\CONVEX{4}{4} \subseteq \SOS{4}{4}$ follows since convex quaternary quartics satisfy the generalized Cauchy-Schwarz inequalities with constants $A_2^* = B_2^* = 1$ by \cref{thm:optimal_cs_constants}.

Let $p$ be a quaternary quartic form satisfying the inequalities in \cref{eq:cs-quartic-1,eq:cs-quartic-2}, and let us prove that $p$ satisfies both requirements appearing in \cref{thm:caracterization-sos}. 

{\bf The first requirement in \eqref{eq:first_implication}.}
Let $\iv_1, \dots, \iv_8 \in \mathbb R^4$ and $\alpha_2, \dots, \alpha_8 \in \{-1, 1\}$ such that $\iv_1\iv_1^T = \sum_{i=2}^8 \alpha_i \iv_i\iv_i^T$. Using the tensor notation developed in \cref{sec:tensors}, this is equivalent to $\iv_1^2 = \sum_{i=2}^8 \alpha_i \iv_i^2$.
Squaring\footnote{The square of a vector $\iv$ is simply the outer product of the vector with itself.} both sides of this equation leads to
\[\iv_1^{4} = \sum_{i=2}^8 \alpha_i\alpha_j \iv_i^{2} \iv_j^{2}.\]
Recall that the biform $(\ix, \iy) \mapsto Q_p(\ix, \iy)$  defined in \cref{eq:biform} can be seen as a linear function of the symmetric outer product $\ix^2\iy^2$. We conclude  that
\[p(\iv_1) = \sum_{2 \le i, j \le 8} \alpha_i \alpha_j Q_p(\iv_i, \iv_j).\]
Using \cref{eq:cs-quartic-1}, we know that $|Q_p(\iv_i,\iv_j)| \le \sqrt{p(\iv_i)p(\iv_j)}$, and therefore
\[p(\iv_1) \le \sum_{2 \le i, j \le 8}\sqrt{p(\iv_i)p(\iv_j)}.\]

{\bf The second requirement in \eqref{eq:second_implication}.}
Let $\iv_3, \dots, \iv_8 \in \mathbb R^4$, $\alpha_3, \dots, \alpha_8 \in \{1, -1\}$ and $\iz \in \mathbb C^4$ such that
$\iz\iz^T + \bar \iz \bar \iz^T = \sum_{i=3}^8 \alpha_i \iv_i\iv_i^T$.
Squaring both side of the equation and applying the biform $Q_p$ as before gives:
\[\sum_{3 \le i,j \le 8} \alpha_i \alpha_j Q_p(\iv_i, \iv_j) = p(\iz) + p(\bar \iz) + 2 Q_p(\iz, \bar \iz) = 2 \RE{p(\iz)} + 2 Q(\iz, \bar \iz).\]
On the one hand, using \cref{eq:cs-quartic-2}, we know that $|p(\iz)| \le Q(\iz,\bar \iz)$, so
\[ 2(\RE{p(\iz)} + |p(\iz)| ) \le 2 \RE{p(\iz)} + 2 Q(\iz,\bar \iz).\]
On the other hand, by \cref{eq:cs-quartic-1},
\[\sum_{3 \le i, j \le 8} \alpha_i \alpha_j Q_p(\iv_i, \iv_j) \le \sum_{3 \le i, j \le 8} \sqrt{p(\iv_i)p(\iv_j)}.\]
In conclusion, $2(|p(\iz)| + \RE{p(\iz)}) \le \sum_{i=3}^8  \sqrt{p(\iv_i)p(\iv_j)}.$

\section{Remarks on the Case of Ternary Sextics}
\label{sec:generalization}
It is natural to ask whether our proof can be extended to show that convex ternary sextics are also sos. \cref{thm:caracterization-sos} for instance generalizes in a straightforward fashion.
\begin{theorem}[\cite{blekherman_nonnegative_2012}]\label{thm:charaterization_sos_ternary_sextics}
  A nonnegative ternary sextic form $p$ is sos if and only if both of the following conditions hold.
  \begin{itemize}
  \item For every $\iv_1, \dots, \iv_9 \in \mathbb R^3$ and $\alpha_2, \dots, \alpha_9 \in \{-1, 1\}$ such that $\iv_1^{ 3} = \sum_{i=2}^9 \alpha_i \iv_i^{ 3}$,
    \begin{equation}
    p(\iv_1) \le \left(\sum_{i=2}^9  \sqrt{p(\iv_i)}\right)^2.
    \end{equation}
  \item For every $\iz \in \mathbb C^3$, for every $\iv_3, \dots, \iv_9 \in \mathbb R^3$, and for every $\alpha_3, \dots, \alpha_9 \in \{-1, 1\}$ such that $\iz^{ 3} + \bar \iz^{ 3} = \sum_{i=3}^9 \alpha_i \iv_i^{ 3}$,
    \begin{equation}
    2 (|p(\iz)| + \RE{p(\iz)}) \le \left(\sum_{i=3}^9 \sqrt{p(\iv_i)}\right)^2.
  \end{equation}
  \end{itemize}
\end{theorem}

In order for us to follow the same proof strategy that applies to quaternary quartics to the set of ternary sextics, we would take an arbitrary convex ternary sextic and try to show that it satisfies both requirements appearing in the previous Theorem. The first requirement is easily dealt with since sextics satisfy the generalized Cauchy-Schwarz inequality appearing in \cref{eq:real_generalized_cauchy_schwarz} with a constant $A_3^*$ equal to 1 (similar to the quartics case). Sextics on the other hand satisfy \cref{eq:complex_generalized_cauchy_schwarz} only with a constant $B_3^*$ strictly larger than 1 (as opposed to $B_2^*=1$ for quartics). This proves to be the main obstacle preventing us from showing that convex ternary sextics satisfy the second requirement in \cref{thm:charaterization_sos_ternary_sextics}. In \cite[Conjecture 7.3]{blekherman_nonnegative_2012}, the author conjectures that this second requirement is actually not needed, in which case our proof strategy would succeed.


\section*{Acknowledgments}
The author is grateful to Amir Ali Ahmadi, Yair Shenfeld, and Ramon van Handel for insightful questions, comments and relevant feedback that improved this draft considerably.

\bibliographystyle{siamplain}
\bibliography{citations}

\appendix
\section{Proof of identity \cref{eq:Qp_integral}}
\label{sec:Qp_integral}
Let $q$ be a $2d\text{-degree}$ form in $n$ variables and let $\ix, \iy$ be two vectors in $\R^n$. By considering the restriction $(x, y) \mapsto q(x\ix + y\iy)$ of the form $q$ to the plane spanned by $\ix$ and $\iy$ if necessary, we can assume without loss of generality that $n = 2$, $\ix = \ie1$, and $\iy = \ie2$.
As a consequence, it suffices to prove that the identity
\[Q_q(\ie1 + i\ie2, \ie1 - i\ie2) = \frac{4^d(d+1)}\pi {2d \choose d}^{-1} \iint_{x^2+y^2 \le 1} q(x, y)  \; {\rm d}x {\rm d}y\]
holds for all bivariate convex forms $q$ of degree $2d$. This identity will follow from the following lemma.

\begin{lemma}\label{lem:integral_laplacian}
  For $k \in \N$, any form $p$ in $\FORMS{2}{k}$ satisfies
  \[\iint_{x^2+y^2 \le 1} \Delta p(x, y) \; {\rm d}x {\rm d}y = k (k+2) \iint_{x^2+y^2 \le 1} p(x, y) \; {\rm d}x {\rm d}y.\]
\end{lemma}

Indeed, using this lemma inductively on the iterates $q, \Delta q, \dots, \Delta^{d-1} q$, we get
\[\iint_{x^2+y^2 \le 1} \Delta^d q(x, y) \; {\rm d}x {\rm d}y = 4^d (d+1) d!^2 \iint_{x^2+y^2 \le 1} q(x, y) \; {\rm d}x {\rm d}y.\]
Since $\Delta^dq$ is a constant and the area of the unit disk is $\pi$, we get that
\[\Delta^d q = \frac{4^d(d+1) d!^2}{\pi}\iint_{x^2+y^2\le 1} q(x, y) \; {\rm d}x {\rm d}y.\]
Recall from \cref{sec:inner_product} that $Q_q(\ie1 + i\ie2, \ie1 - i\ie2) = \frac{1}{(2d)!}\Delta^d q$, and therefore
\[Q_q(\ie1 + i\ie2, \ie1 - i\ie2) = \frac{4^d(d+1)}\pi {2d \choose d}^{-1} \; \iint_{x^2+y^2\le 1} q(x, y) \; {\rm d}x {\rm d}y,\]
which concludes the proof.

\begin{proof}[Proof of \cref{lem:integral_laplacian}]
  Fix $k \in \N$ and a form $p \in \FORMS{2}{k}$. Denote by ${\mathcal D}$ (resp. $\partial {\mathcal D}$) the unit disk (resp. unit circle).
  The well-known divergence theorem states that
  \[\iint_{\mathcal D} \Delta p(x, y) \; {\rm d}x {\rm d}y = \oint_{\partial {\mathcal D}}  \begin{pmatrix}x\\y\end{pmatrix}^T \nabla p(x, y),\]
where $\oint_{\partial {\mathcal D}}$ stands for the line integral over $\partial {\mathcal D}$. Euler's identity shows that the integrand on the right-hand side of the previous equation is $k p(x, y)$, and therefore
\[\iint_{\mathcal D} \Delta p(x, y) \; {\rm d}x {\rm d}y = k \oint_{\partial {\mathcal D}}  p(x, y).\]
Exploiting the fact that the function $p$ is homogeneous of degree $k$ again to relate the integral on ${\mathcal D}$ to the line integral over $\partial {\mathcal D}$ (see \cite[Corollary 1]{integration_hom_sphere_1997}) leads to
\[\oint_{\partial {\mathcal D}}  p(x, y) = (k+2) \iint_{{\mathcal D}}  p(x, y) \; {\rm d}x {\rm d}y,\]
which concludes the proof.
\end{proof}

\section{Proof that the constant $A_d^*$ defined in \eqref{eq:Ad_Bd_star} is larger than $1$ for all even integers $d \ge 4$}
\label{sec:A_d_even}
In this section, we will show that for all even integers $d \ge 4$, there exists a convex bivariate form $p_d$ of degree $2d$ that satisfies $p_d(1, 0) = p_d(0, 1) = 1$ and $Q_{p_d}(\ie1, \ie2) > 1$. This shows that $A_d^* > 1$.

Fix an integer $d \ge 4$ and let $p_d \coloneqq s + \alpha_d \; q$, where
\[s(x, y) \coloneqq  \frac{(x+y)^{2d} + (x-y)^{2d}}2, q(x,y) \coloneqq \sum_{k=1}^{d-1} x^{2k}y^{2d-2k},\]
and $\alpha_d$ is a positive constants defined explicitly in \cref{eq:alpha}. Note that $p_d(1, 0) = p_d(1, 0) = 1$ and $Q_{p_d}(\ie1, \ie2) = 1 + \frac{\alpha_d}{{2d \choose d}} > 1$. 

It remains to prove that the form $p_d$ is convex. The idea of the proof is as follows.
On the one hand, the Hessian of the form $s$ is positive definite everywhere except on the two lines $y = \pm x$, where it is only positive semidefinite. On the other hand, the Hessian of the form $q$ is positive definite on the two lines $y = \pm x$.  By picking $\alpha_d$ to be small enough, we can therefore make the form $p_d$ convex.

More formally, by homogeneity, it suffices to prove that the Hessian of $p$ is positive semidefinite on the circle $\mathcal S \coloneqq \{ (x,y) \in \R^2 \;|\;x^2 +y^2 = 2\}$. Let us now examine the Hessians of the forms $s$ and $q$ individually. The Hessian of $s$ is given by
$$\nabla^2 s(x, y)=d(2d-1)\begin{pmatrix}1&-1\\1&1\end{pmatrix}  \begin{pmatrix}(x+y)^{2d-2}&0\\0&(x-y)^{2d-2}\end{pmatrix} \begin{pmatrix}1&1\\1&-1\end{pmatrix}.$$
The matrix $\nabla^2 s(x, y)$ is positive definite for every $(x, y) \in \mathcal S$ except on the four points $X \coloneqq \{(\pm 1, \pm 1)\}$ where it is only positive semidefinite. We will now prove that the Hessian of $q$ is positive definite on $X$. A simple computation shows that
  \[\nabla^2 q(1, 1) =  \nabla^2 q(-1, -1) = \frac{d(d-1)}3 \begin{pmatrix}4d-5&2d+2\\2d+2&4d-5\end{pmatrix},\]
  \[\nabla^2 q(1, -1) = \nabla^2 q(-1, 1) = \frac{d(d-1)}3 \begin{pmatrix}4d-5&-2d-2\\-2d-2&4d-5\end{pmatrix}.\]
  By examining the trace and the determinant of these matrices (which are univariate polynomials in the variable $d$), we see that they are positive definite if and only if $d \ge \frac72$.
  Let us now partition the circle $\mathcal S$ as
  \[\mathcal S = U \cup (S \setminus  U),\]
  where $U$ is any open subset of $\mathcal S$ containing $X$ on which the matrix $\nabla^2 q$ is positive definite. If we take
  \begin{equation}
  \alpha_d \coloneqq \min_{\|\iu\| = 1, (x, y) \in S \setminus U}\quad \frac{\iu^T\nabla^2s(x, y)\iu}{|\iu^T\nabla^2 q(x, y)\iu|} > 0,\label{eq:alpha}  
\end{equation}
then the Hessian of the form $p_d \coloneqq s + \alpha_d q$ is positive semidefinite on $\mathcal S$, and the form $p_d$ itself is therefore convex.

\section{Simplifying the expression $\sum_{j=0}^{d-1} \sin^2\left(\frac{j\pi}{d}-\theta\right) \cos^{2d-2}\left(\frac{j\pi}{d}-\theta\right)$}\\
\label{app:trigo}
Fix $d \in \N$ and $\theta \in \R$. To simplify notation, let
\[f_d(\theta) \coloneqq \sum_{j=0}^{d-1} \sin^2\left(\frac{j\pi}{d}-\theta\right) \cos^{2d-2}\left(\frac{j\pi}{d}-\theta\right).\]
For $j \in \mathbb N$, let $r_j \coloneqq e^{-i \frac{j \pi}{d}}$. Using the fact that
\[\cos\left( \frac{j\pi}{d}-\theta\right) = \frac{e^{i \theta} r_j + e^{-i \theta}\bar{r_j}}2 \text{ and }\sin\left(\frac{j\pi}{d}-\theta\right) = \frac{e^{i \theta}r_j - e^{-i \theta}\bar{r_j}}{2i},\]
we get that
\begin{align*}
  f_d(\theta) &= -\frac{1}{2^{2d}} \sum_{j=0}^{d-1} \left(e^{i \theta}r_j - e^{-i \theta}\bar{r_j}\right)^2  \left(e^{i \theta}r_j + e^{-i \theta}\bar{r_j}\right)^{2d-2}.
\end{align*}

By expanding and exchanging the order of the summation, we get
\[f_d(\theta) = -\frac{1}{2^{2d}} \sum_{h=0}^{2d-2} {2d-2\choose h}  \left( (e^{i\theta})^{2h}\sum_{j=0}^{d-1} r_j^{2h}
  +  (e^{i\theta})^{2h-4}\sum_{j=0}^{d-1} r_j^{2h-4} 
  -2  (e^{i\theta})^{2h-2} \sum_{j=0}^{d-1} r_j^{2h-2}  \right).\]
We now use the following simple fact about the sum of the $k^{\text{th}}$ powers of roots of unity:
\[\forall k \in \mathbb N \quad \sum_{j=0}^{d-1} r_j^{2k}  = \left\{ \begin{array}{ll}d & \text{if } d \text{ divides }  k \\ 0& \text{otherwise,}\end{array}\right. \]
to get 
\begin{align*}
  f_d(\theta)
  &= -\frac{d}{2^{2d}} \sum_{h=0}^{2d-2} {2d-2\choose h} \left( e^{2i \theta h} 1_{\{d \; | \; h\}} + e^{ 2i (h-2) \theta }  1_{\{d  \; | \;  h-2\}} - 2 e^{2i(h-1)\theta}  1_{\{d \; | \; h-1\}}\right),
\end{align*}
and therefore
\[f_d(\theta)= \frac{2d}{2^{2d}}  \left(\frac {{2d-2\choose d-1}}d - \cos(2d\theta)\right).\]

\section{Hessian and Laplacian in polar coordinates}\label{sec:proof_hessian_laplacian_polar}
  \begin{proof}[Proof of \cref{lem:diff_in_polar_coordinates}]
    Fix a positive integer $k$ a bivariate form $p \in
    \FORMS{2}{k}$. Let us switch from cartesian coordiantes $(x, y)$
    to polar coordinates $(r, \theta)$ defined by
    \(x = r\cos(\theta)\) and \(y =r\sin(\theta)\) and write
    $p(x, y) \eqqcolon r^k f(\theta)$, for some twice-differentiable
    function $f$.  Recall that the gradient operator $\nabla$ can be
    written in polar coordinates as follows
  \[\nabla = \frac{\partial}{\partial r} \ie{r} + \frac 1r
  \frac{\partial}{\partial \theta} \ie\theta, \text{ where } \ie{r}
  \coloneqq \begin{pmatrix}\cos(\theta)\\\sin(\theta)\end{pmatrix}
  \text{ and } \ie{\theta}
  \coloneqq \begin{pmatrix}-\sin(\theta)\\\cos(\theta)\end{pmatrix}.\]
  The Hessian operator $\nabla^2 = \nabla \cdot \nabla^T$ is thus given by
  \[\nabla^2 = \frac{\partial^2}{\partial r^2} \ie{{rr}} + \frac{\partial}{\partial r}\left(\frac 1r \frac{\partial}{\partial \theta}\right)  \ie {r\theta} + \left(\frac1r \frac{\partial}{\partial r} + \frac 1{r^2} \frac{\partial^2}{\partial \theta^2} \right) \ie{\theta\theta},\]
  where $\ie{rr} = {\ie r} {\ie r}^T$, $\ie{r\theta} = {\ie r} {\ie\theta}^T +  {\ie\theta}{\ie r}^T$ and $\ie{\theta\theta} = {\ie \theta} {\ie\theta}^T$.
  Note that taking the derivative of a form of degree $k' \ge 1$ with respect to $r$ is equivalent to multiplying by $\frac{k'}r$.
  The Hessian operator, when applied to the $k\text{-degree}$ form $p$, can thus be simplified further to
  \[\nabla^2p(x,y) = r^{k-2} \; \left( k(k-1) f(\theta) \ie{rr} + (k-1)
    f'(\theta) \ie {r\theta} + \left(k +
      f''(\theta) \right) \ie {\theta \theta
    }\right).\]
  The Laplacian $\Delta p$ is given by the trace of the matrix
  $\nabla^2p$. Since the trace of both matrices $\ie {r r}$ and $\ie {\theta \theta}$ is one and the trace of $\ie {r \theta}$ is zero, we get that
  \[\Delta p(x, y) = r^{k-2}\; \left( k^2 f(\theta) + f''(\theta)\right).\]
\end{proof}

\section{Proof of \cref{lem:holder}}\label{app:proof_lemma}
Let us first prove that for all nonnegative scalars $x_2, \dots, x_n,$
the optimal value of the  minimization problem below is equal to $(\sum_{i=2}^n \sqrt{x_i})^2$.
\[\min \sum_{i=2}^n a_ix_i \quad \text{ s.t. } \quad a_i > 0  \text{ for }
  i=2,\dots,n \quad\text{ and } \quad\sum_{i=2}^n \frac1{a_i} = 1 .\]
Let $\gamma$ stand for the optimal value of this optimization
problem. Taking
$a_i = \left(\sum_{j=2}^n \sqrt{x_j} \right) x_i^{-\frac 12}$ for
$i = 2, \dots, n$ (with the convention that $0^{-1} = +\infty$) shows
that $\gamma \le (\sum_{i=2}^n \sqrt{x_i})^2$. We now show that
$\gamma \ge (\sum_{i=2}^n \sqrt{x_i})^2$. Consider positive scalars
$a_2, \dots, a_n$ satisfying $\sum_{i=2}^n \frac1{a_i} = 1$.  Note
that
\[\sum_{i=2}^n \frac1{a_i} =  {\bf 1}^TA^{-1}{\bf 1}, \]
where ${\bf 1}^T \coloneqq (1, \dots, 1) \in \mathbb R^{n-1}$ and $A$ is the diagonal $(n-1) \times (n-1)$ matrix with the $a_i$ as diagonal elements.
By taking the Schur complement, the inequality $1 - {\bf 1}^TA^{-1}{\bf 1} \ge
0$ implies that $A \succeq {\bf1 1}^T$. Therefore, by  multiplying each
side of this matrix inequality by $\iu^T
\coloneqq (\sqrt{x_2}, \dots, \sqrt{x_n})$, we get
\(({\bf1}^T \iu)^2  \le \iu^T A \iu,\)
i.e.,
\((\sum_{i=2}^n \sqrt{x_i})^2 \le (\sum_{i=2}^n a_ix_i) ^2.\)  In
conclusion, $\gamma \ge  (\sum_{i=2}^n \sqrt{x_i})^2$.

Let us now prove that for any complex number $z$,
\[\max_{a \in \mathbb C, \frac1a + \frac1{\bar a} = 1} az + \bar a \bar z =  2(|z| + \RE{z}).\]
First notice that $a \in \mathbb C$ satisfies $\frac1a + \frac1{\bar a} = 1$ if and only if $a$ has the form $2\cos(\theta)e^{i\theta}$ for some $\theta \in \mathbb R$.

Write $z = |z| e^{i\alpha}$ for some $\alpha \in \mathbb R$. Then,
  \begin{align*}
    \max_\theta \cos(\theta) \RE{e^{i\theta} z}
    &= \max_\theta |z|\cos(\theta) \cos(\theta+\alpha)
    \\&= \frac12 |z| \max  (\cos(\alpha) + \cos(2\theta+\alpha))
    \\&=  \frac12 |z| (1+\cos(\alpha))
    \\&= \frac{|z| + \RE{z}}2.
  \end{align*}
The result follows as $az+\bar a\bar z = 4\RE{\cos(\theta) e^{i \theta}z}$.

\end{document}